\documentclass[12pt]{amsart}
\usepackage{amsmath}
\usepackage[active]{srcltx}
\usepackage{t1enc}
\usepackage[latin2]{inputenc}
\usepackage{verbatim}
\usepackage{amsmath,amsfonts,amssymb,amsthm}
\usepackage[mathcal]{eucal}
\usepackage{enumerate}
\usepackage[centertags]{amsmath}
\usepackage{graphics}
\numberwithin{equation}{section}

\usepackage{txfonts} 
\usepackage[T1]{fontenc}
\usepackage{lmodern}

\usepackage{mathtools}
\mathtoolsset{showonlyrefs,showmanualtags} 

\usepackage[msc-links]{amsrefs}  

\usepackage{hyperref} 
\hypersetup{
	colorlinks=true,       
	linkcolor=blue,          
	citecolor=magenta,        
	filecolor=magenta,      
	urlcolor=cyan           
}

\newtheorem{OldTheorem}{Theorem}

\newtheorem{theorem}{Theorem}
\newtheorem{lemma}{Lemma}
\newtheorem{definition}{Definition}
\newtheorem{corollary}{Corollary}
\newtheorem{prop}{Proposition}
\newtheorem{remark}{Remark}
\def\ll{\left }
\def\rr{\right }

\def\dist{{\rm dist\,}}

\def\spec{{\rm spec\,}}

\def\ZZ{\ensuremath{\mathbb Z}}

\def\zI{\ensuremath{\mathcal I}}
\def\ZN{\ensuremath{\mathbb N}}
\def\ZT{\ensuremath{\mathbb T}}
\def\ZI{\ensuremath{\textbf I}}

\def\md#1#2\emd{\ifx0#1
\begin{equation*} #2 \end{equation*}\fi  
\ifx1#1\begin{equation}#2\end{equation}\fi   
\ifx2#1\begin{align*}#2\end{align*}\fi   
\ifx3#1\begin{align}#2\end{align}\fi    
\ifx4#1\begin{gather*}#2\end{gather*}\fi  
\ifx5#1\begin{gather}#2\end{gather}\fi   
\ifx6#1\begin{multline*}#2\end{multline*}\fi  
\ifx7#1\begin{multline}#2\end{multline}\fi  
}
\newcommand {\e }[1]{\eqref{#1}}
\newcommand {\lem }[1]{Lemma \ref{#1}}
\newcommand {\trm }[1]{Theorem \ref{#1}}
\newcommand {\cor }[1]{Corollary \ref{#1}}
\newcommand {\pro }[1]{Proposition \ref{#1}}

\begin{document}
	\title[Menshov type correction theorems]{Menshov type correction theorems for sequences of compact operators}

	\author{Grigori A. Karagulyan}
	\address{Faculty of Mathematics and Mechanics, Yerevan State
		University, Alex Manoogian, 1, 0025, Yerevan, Armenia} 
	\email{g.karagulyan@ysu.am}
	
\subjclass[2010]{42A20, 42B16, 47A58, 47B07}
	\keywords{Menshov correction theorems, compact operators, almost everywhere convergence, Fourier series, Walsh system}
	\date{}
	\maketitle
	\begin{abstract}
		We prove Menshov type "correction" theorems for sequences of compact operators, recovering several results of  Fourier series in trigonometric and Walsh systems. The paper clarifies the main ingredient, which is important in the study of such "correction" theorems.  That is the weak-$L^1$ estimate for the  maximal Fourier sums of  indicator functions of some specific sets.
	\end{abstract}
\section{Introduction}
The modifications of functions in order to improve convergence properties of their  Fourier series is an old issue in Fourier Analysis. A well known modification method is the change of function values on a set of small measure.  Menshov's two classical theorems (\cite{Men1}, \cite{Men2}, see also \cite{Bari}) were crucial in this study. 
\begin{OldTheorem}[Menshov, \cite{Men1}]\label{Menshov-1}
	For any continuous function $f\in C(\ZT)$ and $\varepsilon>0$ there is a function $g\in C(\ZT)$, whose trigonometric Fourier series is uniformly convergent and $|\{f(x)\neq g(x)\}|<\varepsilon$.
\end{OldTheorem}
Observe that in the statement of this theorem the initial function $f$ can be equivalently taken to be arbitrary finite-valued measurable function, and this follows from the well known theorem of Luzin on continuous modification of measurable functions. Besides, one can see that in this theorem the modification set $\{f\neq g\}$ depend on the initial function. In the next theorem of Menshov the initial function is modified on a given everywhere dense open set, but the resulting function gets almost everywhere convergence Fourier series instead of uniformly.
\begin{OldTheorem}[Menshov, \cite{Men2}]\label{Menshov-2}
	Let $f\in L^1(\ZT)$ and $G\subset \ZT$ be an everywhere dense open set. Then there is a function $g\in L^1(\ZT)$ with an a.e. convergent Fourier series so that $\{g(x)\neq f(x)\}\subset G$.
\end{OldTheorem}
An elegant prove of \trm {Menshov-1} was given by Olevskii in \cite{Ole}, where one can also find a nice review of some other related results. Extensions of \trm {Menshov-1} for Walsh and other multiplicative systems were proved in the papers \cite {GoWa}, \cite{BaRu}, \cite{Khr}. The papers \cite{Kash} and \cite{KaKo} consider the analogous of \trm {Menshov-1} for trigonometric and general orthonormal matrices. 

It is well known that one can not claim $L^1$-norm convergence in \trm {Menshov-2} instead of a.e.. Nevertheless, Grigoryan \cite{Grig2} proved existence of an open set $G$ of small measure, serving as a correction set for $L^1$-convergence of Fourier series. Namely,
\begin{OldTheorem}[Grigoryan, \cite{Grig2}]\label{Grig2}
	For any $\varepsilon>0$ there exists an open set $G_\varepsilon\subset (\ZT)$ with $|G_\varepsilon|<\varepsilon$ such that for any $f\in L^1(\ZT)$ one can find a function $g\in L^1(\ZT)$, whose Fourier series converges in $L^1$-norm and $\{g(x)\neq f(x)\}\subset G$.
\end{OldTheorem}
Note that the full statement of this theorem in \cite{Grig2} provides also some control on the Fourier coefficients of the resulting function $g$.  
Grigoryan \cite{Grig3} extended the result of \trm {Grig2} for complete orthonormal systems of bounded functions. The following result of Grigoryan-Navasardyan is a version of \trm {Grig2} for Walsh system.
\begin{OldTheorem}[Grigoryan, Navasardyan, \cite{GrNa}]\label{GrNa}
	For any $\varepsilon>0$ there exists an open set $G_\varepsilon\subset [0,1]$ with $|G_\varepsilon|<\varepsilon$ such that for any $f\in L^1[0,1]$ one can find a function $g\in L^1[0,1]$, whose Walsh-Fourier series converges in $L^1$-norm, $\{g(x)\neq f(x)\}\subset G_\varepsilon$, and the sequence of absolute values of the Fourier-Walsh coefficients of $g$  is decreasing.
\end{OldTheorem}
A likewise problem for almost everywhere convergence with a weaker monotonicity condition on the Fourier coefficients was considered in \cite{Grig4} (see also \cite{Grig1}). That is 
\begin{OldTheorem}[Grigoryan, \cite{Grig4}]\label{Grig4}
	For any $\varepsilon>0$ there exists an open set $G_\varepsilon\subset [0,1]$ with $|G_\varepsilon|<\varepsilon$ such that for any $f\in L^1[0,1]$ one can find a function $g\in L^1[0,1]$, whose Walsh-Fourier series is a.e. convergent, $\{g(x)\neq f(x)\}\subset G_\varepsilon$, and the sequence of absolute values of non-zero Fourier-Walsh coefficients of $g$  is decreasing.
\end{OldTheorem}
Grigoryan-Sargsyan \cite{GrSa} recently proved the analogous of \trm{Grig4} for the Vilenkin systems of bounded type.

In this note we consider similar problems for sequences of compact operators with additional properties that are common for the partial sum operators of Fourier series. To state the main results we need some definitions and notations. 
Intervals in our definitions are the intervals of the form $[a,b)\subset [0,1)$. A set $G\subset [0,1)$ is said to be a finite-interval set if it is a union of finite number of intervals. The indicator function of a set $G$ will be denoted by $\ZI_G$.
\begin{definition}
	A sequence of functions $f_n\in L^1(0,1)$ is said to be weakly convergent to a function $f\in L^1(0,1)$ if 
	\begin{equation*}
\lim_{n\to\infty}\int_0^1f_ng=\int_0^1fg
	\end{equation*}
	for any $g\in L^\infty(0,1)$. We denote this relation by  $f_n\overset{w}{\to} f$.
\end{definition}
\begin{definition}
	A bounded linear operator $U: L^1(0,1)\to L^1 (0,1)$ is said to be  compact if $\|U(f_n)-U(f)\|_1\to 0$ whenever $f_n\overset{w}{\to} f$.
\end{definition}
\begin{definition}
	A countable family $\zI$ of intervals is said to be basis if  
	
	1) $[0,1)\in \zI$,
	
	2) for any $\Delta=[a,b)\in \zI$ there are infinitely many integers $l>0$ such that $[a+|\Delta|(j-1)/l,a+|\Delta|j/l)\in\zI$, $j=1,2,\ldots, l$. 
	
\end{definition}

For a sequence of bounded linear operators
\begin{equation}\label{oper}
U_n: L^1(0,1)\to L^\infty (0,1),\quad n=1,2,\ldots,
\end{equation}
we denote 
\begin{equation*}
U^*f(x)=\sup_{n}|U_nf(x)|.
\end{equation*}
We shall consider operator sequences \e {oper} satisfying the following properties, where $1<p<\infty$:
\begin{enumerate}
	\item [(A)] each $U_n$ is a compact operator,
	\item [(B)] $\|U_n(f)-f\|_p\to 0$ as $n\to\infty$ for every $f\in L^p$,
	\item [(C)] $U_n(\ZI_{G})$ converges almost everywhere for any interval $G$,
	\item [(D)] for any $0<\varepsilon<1$ there is a sequence of finite-interval sets $G_l=G_l(\varepsilon)\subset[0,1)$, $l=1,2,\ldots,$ such that  
	\begin{align}
	&\alpha\varepsilon\le |G_l|\le \varepsilon,\label{a58}\\
	&\, |G_l|^{-1}\cdot\ZI_{G_l}\overset{w}{\to}\ZI_{[0,1)}  \text{ as } l\to\infty,\label{a27}\\
	\end{align}
	and there is a basis $\zI$ so that for any $\Delta\in\zI$ and $\lambda>0$ we have
	\begin{align}
	&\, \lambda\cdot |\{U^*(\ZI_{G_l\cap \Delta})>\lambda\}|\le \beta\cdot|G_l\cap\Delta|,\quad l=1,2,\ldots,\label{a23}
	\end{align}
where $0<\alpha<1 $ , $\beta>0$ are  constant depended only on $U_n$.
\end{enumerate}
\begin{theorem}\label{T1}
	Let an operator sequence  $U_n:L^1(0,1)\to L^\infty (0,1)$ satisfy the properties (A), (B) and $\varepsilon >0$. Then there is an open set $G_\varepsilon\subset (0,1)$ such that $|G_\varepsilon|<\varepsilon$ and for every function $f\in L^1(0,1)$  one can find a $g\in L^1(0,1)$ with $\{f(x)\neq g(x)\}\subset G$ and  satisfying
	\begin{equation*}
	\|U_ng -g\|_1\to 0.
	\end{equation*}
	If $U_n$ satisfies also conditions $(C)$ and (D), then we will additionally have 
	\begin{equation*}
	U_ng(x)\to g(x)\text { almost everywhere}.
	\end{equation*}
\end{theorem}
A slight change in the statement of property (D) allows us to prove the full analogous of Menshov's \trm {Menshov-2} for sequences of compact operators. So in the next theorem instead of property (D) we shall suppose that
\begin{enumerate}
	\item [(D*)]  for every $0<\varepsilon<1$ and everywhere dense open set $U$ there is a sequence of finite-interval sets $G_l\subset U$, $l=1,2,\ldots,$ such that $|G_l|\le \varepsilon$ and there hold relations \e {a27} and \e {a23}.
\end{enumerate}
\begin{theorem}\label{T2}
	Let an operator sequence  $U_n:L^1(0,1)\to L^\infty (0,1)$ satisfy the properties (A), (B), (C), and (D*) and $U\subset [0,1)$ be an everywhere dense open  set. Then for every $f\in L^1(0,1)$  there is a function $g\in L^1(0,1)$ such  that $g(x)=f(x)$, $x\in G$, and 
	\begin{equation*}
	U_ng(x)\to g(x)\text { almost everywhere}.
	\end{equation*}
\end{theorem}

\begin{corollary}\label{C1}
	Let $f\in L^1[0,1]$ and $G\subset [0,1]$ be an everywhere dense open set. Then there is a function $g\in L^1[0,1]$, whose Fourier series in Walsh system (in a given bounded type Vilenkin system) converges a.e. and $\{g(x)\neq f(x)\}\subset G$.
\end{corollary}
\begin{corollary}\label{C2}
	Let $\{\phi_n\in L^\infty(0,1):\, n=1,2,\ldots\}$ be a basis in $L^p(0,1)$, $1<p<\infty$,   and $\varepsilon>0$. Then there exists an open set $G\subset (0,1)$ with $|G|<\varepsilon$ such that for any $f\in L^1(0,1)$ one can find a function $g\in L^1(0,1)$, whose Fourier series in $\{\phi_n\}$ convergence in $L^1$-norm and $\{g(x)\neq f(x)\}\subset G$.
\end{corollary}
It is well known that the partial sum operators of Fourier series in classical orthogonal systems satisfy the properties (A), (B) and (C), while the weak-$L^1$-condition \e {a23} is more delicate. We will see in the last section that properties (D) and (D*) are satisfied for trigonometric, Walsh and for the bounded type Vilenkin systems. 
The proof of those properties are based on the corresponding propositions (\pro{P1}, \pro{P2}) showing weak type estimates for the maximal partial sum operators of indicator functions of "uniformly distributed"  finite-interval sets. These results are interesting itself.  The trigonometric case is more delicate.  These propositions clarify the main ingredient, which is important in the study of such "correction" theorems. 

Hence, \cor {C1} immediately follows from the combination of \trm {T2} and \pro{P22}.  Likewise, the combination of \trm {T2} and \pro {P11} implies Menshov's \trm {Menshov-2}.

As for \cor {C2}, which is the extension of the analogous theorem for complete orthonormal systems from \cite {Grig3}, it immediately follows from \trm {T1}. 

Finally, note that \trm {T1} partially implies \trm {Grig2}, \trm {GrNa},  \e {Grig4} as well as some other results of papers \cite {Grig1, Grig3, GrSa}, without claiming the monotonicity conditions of the Fourier coefficients. 

\section{Proof of Theorems}
Before to state the main lemma we will need the following:

\begin{remark}
	An example of sets $G_l(\varepsilon)$ satisfying \e {a58} and \e {a27} is very simple. One can chose 
	\begin{equation}\label{a29}
	G_l(\varepsilon)=\bigcup_{k=0}^{l-1}\left[\frac{k}{l},\frac{k+\varepsilon}{l}\right).
	\end{equation}
\end{remark}

\begin{remark}
	If the operators $U_n$ satisfy properties (A), (C) and sequence of finite interval sets $G_n$ satisfy \e {a27} and  \e {a23}, then we will have the same weak type inequality \e {a23} for any $G\in\zI$. Indeed, given  $\lambda>0$, using a.e. convergence of $U_n(\ZI_G)$, one can find an integer $m$ such that
\begin{equation}\label{a43}
|\{\sup_{n>m}\ll|U_n(\ZI_{G})\rr|>\lambda\}|<|G|/\lambda.
\end{equation}
The compactness of $U_n$ and the weak convergence property \e {a27} easily yields  
\begin{align*}
&|G_k|^{-1}\cdot|G_k\cap G|\to |G|\text { as }k\to\infty,\\
&\ll\|U_n\ll(|G_k|^{-1}\cdot \ZI_{G_k\cap G}\rr)-U_n(\ZI_G)\rr\|_1\to 0\text { as }k\to\infty. 
\end{align*}
Thus, applying also \e {a23},  we get
\begin{align}\label{a44}
|\{\sup_{1\le n\le m}\ll|U_n(\ZI_{G})\rr|>\lambda\}|&\le \limsup_{k\to\infty}\ll|\ll\{\sup_{1\le n\le m}\ll|U_n\ll(\frac{\ZI_{G_k\cap G}}{|G_k|}\rr)\rr|>\lambda\rr\}\rr|\\
&\le  \limsup_{k\to\infty}\ll|\ll\{U^*\ll(\frac{\ZI_{G_k\cap G}}{|G_k|}\rr)>\lambda\rr\}\rr|\\
&\le \limsup_{k\to\infty}\left(\frac{\beta|G_k|^{-1}|G_k\cap G|}{\lambda}\right)= \frac{\beta|G|}{\lambda}.
\end{align}
Combining \e {a43} with this, we will get 
\begin{equation*}
|\{U^*(\ZI_{G})>\lambda\}|\lesssim|G|/\lambda.
\end{equation*}
\end{remark}
Here and below the notation $a\lesssim b$ will stand for the inequality $a\le c\cdot b$, where $c>0$ is a constant depended only on the parameters of the operator sequence $U_n$ that can appear in the statements of properties (A)-(D*). We shall say $f$ is a step function, if it can be written in the form
\begin{equation}\label{step}
f(x)=\sum_{k=1}^la_k\ZI_{\Delta_k}(x),\quad a_k\neq 0,\quad k=1,2,\ldots, l,
\end{equation}
where $\Delta_k$ are are pairwise disjoint intervals and we say it is a $\zI$-step function if each $\Delta_k$ is from a given basis $\zI$. 
\begin{lemma}\label{L1}
Let a sequence of bounded linear operators \e {oper} satisfy conditions (A) and (B). Then for any choice of numbers $0<\varepsilon, \eta<1$ and a step function $f(x)$ there is a step function $g(x)$ such that 
\begin{align}
&|\{g(x)\neq f(x)\}|\le \varepsilon,\label{a47}\\
&\|U_n(g)\|_1\lesssim\|g\|_1\le 2 \|f\|_1,\quad n=1,2,\ldots .\label{a6}
\end{align}
If in addition $U_n$ satisfies (C) , (D) and $f$ is a $\zI$-step function, then we will also have 
\begin{equation}\label{a11}
 t\cdot |\{U^*(g)>t\}|\lesssim \|f\|_1,\quad \eta<t<1.
\end{equation}
\end{lemma}
\begin{proof}  
First we shall prove the basic part of the lemma, supposing that all conditions (A), (B), (C) and (D) hold simultaneously.  Hence, let $f(x)$ be a $\zI$-step function of the form \e{step}. By the definition of $\zI$ each interval from $\zI$ can be split into smaller  intervals from $\zI$. Thus we can suppose that all $\Delta_k$ in \e{step} satisfy
\begin{equation}\label{delta}
\Delta_k\in \zI,\quad |\Delta_k|<\delta,\quad k=1,2,\ldots, l,
\end{equation}
where $\delta>0$ can be arbitrarily small. Chose a sequence of finite-interval sets $G_m$, satisfying the conditions of property (D) corresponding to a number $\varepsilon/2$. Using weak convergence property \e {a27}, we have $|G_m\cap \Delta_k|/|G_m|\to|\Delta_k|$ for any $k=1,2,\ldots ,l$. Thus, applying \e {a58}, one can check that the sets  $G_m^{(k)}=G_m\cap \Delta_k$ can satisfy the inequalities
\begin{equation}\label{a45}
\frac{\alpha\varepsilon}{4}|\Delta_k|\le |G_m^{(k)}|\le \varepsilon|\Delta_k| ,\quad m>m_0,
\end{equation}
and we have
\begin{align}
\lambda_m^{(k)}(x)=a_k\left(\ZI_{\Delta_k}(x)- \frac{|\Delta_k|\cdot \ZI_{G_m^{(k)}}(x)}{|G_m^{(k)}|}\right)\overset{w}{\to} 0\text { as } m\to\infty,\label{a12}
\end{align}
for any fixed $k$, where $a_k$ are the coefficient from \e{step}.  Applying \e{delta} and lower bound in \e {a45}, one can easily check that $\|\lambda_m^{(k)}\|_p\lesssim \max |a_k|\cdot (\delta/\varepsilon^{p-1} )^{1/p}$ and so we can fix a smaller enough $\delta$ in \e{delta} to ensure 
\begin{equation}\label{a14}
\|\lambda_m^{(k)}\|_p\le \|f\|_1,\quad k,m=1,2,\ldots.
\end{equation}
Using \e{a23} and the remark before lemma, we conclude 
\begin{align}\label{a46}
|\{U^*(\lambda_m^{(k)})>\lambda\}|
& \le |\{U^*(a_k\ZI_{\Delta_k})>\lambda/2\}|\\
&\qquad +\ll|\ll\{U^*\ll(\frac{a_k|\Delta_k|\cdot \ZI_{G_m^{(k)}}(x)}{|G_m^{(k)}|}\rr)>\lambda/2\rr\}\rr|\\
&\lesssim \frac{ |a_k||\Delta_k|}{\lambda}+\frac{|a_k||\Delta_k||G_m^{(k)}|}{\lambda |G_m^{(k)}|}\lesssim\frac{ |a_k||\Delta_k|}{\lambda}.
\end{align}
For any number $\xi>0$ one can inductively construct integers $1\le N_1<\cdots <N_l$ and $1\le m_1<m_2<\cdots <m_l$ such that 
\begin{align}
&\ll\|U_n\ll(\lambda_{m_j}^{(j)}\rr)-\lambda_{m_j}^{(j)}\rr\|_1< \xi,\quad n\ge N_j,\quad 1\le j\le l,\label{ll}\\
&\ll|\ll\{\sup_{m\ge N_j}|U_m(\lambda_{m_j}^{(j)})-\lambda_{m_j}^{(j)}|>\frac{\eta}{4l}\rr\}\rr|<\xi,\quad 1\le j\le l,\label{lll}\\
&\ll\|U_n\ll(\lambda_{m_j}^{(j)} \rr)\rr\|_1<\frac{\xi}{N_{j-1}} ,\quad n\le N_{j-1},\quad 2\le j\le l,\label{l}
\end{align}
and those are constructed in the order $m_1, N_1, m_2, N_2, \cdots , m_l, N_l$. For \e {ll} the property (B) is used, while \e {lll} follows from a.e. convergence property (C). The inequality \e {l} is based on the compactness of operators $U_n$ combined with \e{a12}.  We define the desired function by
\begin{equation*}
g(x)=\sum_{j=1}^l\lambda_{m_j}^{(j)}(x)
\end{equation*}
choosing $\xi$ to be small enough number. From \e {a45} and \e {a12} we immediately get
\begin{align}\label{a40}
\|g\|_1\le 2\|f\|_1,\quad 
|\{g(x)\neq f(x)\}|=\sum_{j=1}^l|G_{m_j}^{(j)}|\le \varepsilon.
\end{align}
For our further convenience we set $N_{0}=0$, $N_{l+1}=\infty$, $\lambda_{m_{l+1}}^{(l+1)}\equiv 0$ and assume $\sum_a^b=0$ whenever $a>b$. Applying Banach-Steinhaus theorem, from property (B) we get $\|U_n\|_{L^p\to L^p}\le M$ and so by \e {a14}
\begin{equation}\label{a2}
\ll\|U_n\ll(\lambda_{m_j}^{(j)}\rr)\rr\|_{1}\le \ll\|U_n\ll(\lambda_{m_j}^{(j)}\rr)\rr\|_p\le M\|\lambda_{m_j}^{(j)}\|_p\lesssim \
\|f\|_1.
\end{equation} 
Then, using also \e {ll}, \e {l} and \e {a40}, for 
\begin{equation}\label{a41}
N_{k-1}<n \le N_k,\quad k=1,2,\ldots,l+1,  
\end{equation}
and for a small enough $\xi$ ($\xi<\|f\|_1/l$) we conclude
\begin{align}\label{a48}
\|U_n(g)\|_1&\le \ll\|\sum_{j=1}^{k-1}U_n\ll(\lambda_{m_j}^{(j)} \rr)\rr\|_{1}+\ll\|\sum_{j=k}^{l}U_n\ll(\lambda_{m_j}^{(j)} \rr)\rr\|_{1}\\
&\le\ll\|\sum_{j=1}^{k-1}\lambda_{m_j}^{(j)}\rr\|_{1}+\sum_{j=1}^{k-1}\ll\|U_n\ll(\lambda_{m_j}^{(j)}\rr)-\lambda_{m_j}^{(j)}\rr\|_{1}\\
&\qquad +\sum_{j=k+1}^{l}\ll\|U_n\ll(\lambda_{m_j}^{(j)} \rr)\rr\|_{1}+ \ll\|U_n\ll(\lambda_{m_k}^{(k)} \rr)\rr\|_{1}\\
&\lesssim \|g\|_1+l\xi +\frac{l\xi}{N_k}+\|f\|_1\\
&\le 4\|f\|_1,
\end{align}
that implies \e {a6}. To prove \e {a11} we let $n$ to be an arbitrary positive integer and let it satisfy \e {a41}. We have 
\begin{align*}
|U_n(g)|&\le \sum_{j=1}^l|U_n\ll(\lambda_{m_j}^{(j)} \rr)|\\
&\le\sum_{j=1}^{k-1}\ll|\lambda_{m_j}^{(j)}\rr|+\sum_{j=1}^{k-1}\ll|U_n\ll(\lambda_{m_j}^{(j)}\rr)-\lambda_{m_j}^{(j)}\rr|+\sum_{j=k+1}^{l}\ll|U_n\ll(\lambda_{m_j}^{(j)}\rr)\rr|+|U_n\ll(\lambda_{m_k}^{(k)} \rr)|\\
&\le |g(x)|+\sum_{j=1}^{l}\sup_{m\ge N_j}|U_m(\lambda_{m_j}^{(j)})-\lambda_{m_j}^{(j)}|\\
&\qquad +\sum_{s=1}^{l-1}\sum_{j=s+1}^{l}\sum_{m=N_{s-1}+1}^{N_s}\ll|U_m\ll(\lambda_{m_j}^{(j)}\rr)\rr|+\sup_{1\le j\le l}U^*\ll(\lambda_{m_j}^{(j)}\rr)\\
&=A_1+A_2+A_3+A_4.
\end{align*}
Observe that each of functions $A_i$, $i=1,2,3,4$, is independent of $n$ and so we can write
\begin{equation}\label{a18}
U^*g(x)\le A_1(x)+A_2(x)+A_3(x)+A_4(x).
\end{equation}
For $A_1=|g|$ we write Chebyshev's inequality
\begin{equation}\label{a19}
|\{x\in (0,1):\, A_1(x)>t/4\}|\lesssim \frac{\|g\|_1}{t}\le  \frac{2\|f\|_1}{t},\quad t>0.
\end{equation}
Then, applying \e {lll}, for $t>\eta$ and a small enough $\xi$ ($\xi<\|f\|_1/l$) we get
\begin{align}\label{a20}
|\{x\in (0,1):\,A_2(x)>t /4 \}|&\le |\{x\in (0,1):\,A_2(x)>\eta /4 \}|\\
&\le l\cdot\xi\le  \|f\|_1,\quad t>\eta.
\end{align}
From \e {l} with $\xi<\|f\|_1/l^2$ it follows that 
\begin{equation*}
\left\|A_3\right\|_1\le l^2\cdot (N_s-N_{s-1})\cdot\frac{\xi}{N_s}\le \|f\|_1,
\end{equation*}
then, again writing Chebyshev's inequality, we will get 
\begin{equation}\label{a21}
|\{x\in (0,1):\, A_3(x)>t/4\}|\lesssim\frac{\|f\|_1}{t},\quad t>0.
\end{equation}
Applying \e {a46} we obtain 
\begin{align}\label{a15}
|\{A_4(x)>t/4\}|\}|&\le \sum_{j=1}^l|\{U^*\ll(\lambda_{m_j}^{(j)} \rr)>t/4 \}|\\
&\lesssim \sum_{j=1}^l \frac{|a_j||\Delta_j|}{t}= \frac{\|f\|_1}{t}.
\end{align}
Combining \e {a18}, \e {a19}, \e {a20}, \e {a21} and  \e {a15} , we obtain \e {a11} that completes the proof of basic part of lemma.

Now suppose that only properties (A) and (B) hold and $f$ is an arbitrary step function.  So the lemma claims to find a function $g$ satisfying \e {a47} and \e {a6}. To do it we need to review once again the proof of the basic part of lemma with slight changes described below.  As before we shall consider the step function \e{step}, where $\Delta_k$, $k=1,2,\ldots,l$ are arbitrary interval satisfying \e{delta} for small enough $\delta$, but the sets $G_m^{(k)}$ should be defined differently, that is 
\begin{equation*}
G_m^{(k)}=G_{m}(\varepsilon)\cap \Delta_k,
\end{equation*}
where $G_m(\varepsilon)$ is the set defined in \e {a29}. This gives a slight change in the definition of the function $g$, but $g$ will still satisfy \e {a47} and \e {a6}. To proceed the proof one needs to omit inequalities \e {a46} and \e {lll} obtained from conditions (C) and (D) and neglect the part of the proof concerning the bound \e {a11}. 
\end{proof}

\begin{proof}[Proof of \trm {T1}]  Let $f_k(x)$, $k=1,2,\cdots $ be a sequence of step functions that is everywhere dense in $L^1(0,1)$. In the case of extra conditions of theorem we take $f_k$ to be $\zI$-step function, where  $\zI$ is the basis from the statement of condition (D). Existence of a such sequence  follows from the properties of basis $\zI$. Applying lemma for $\varepsilon_k=\varepsilon\cdot 2^{-k}$ and $\eta_k=4^{-k}$, we find a sequence of step functions $g_k(x)$ and sets $E_k\subset (0,1)$ such that
\begin{align}
&|\{f_k(x)\neq g_k(x)\}|<\varepsilon \cdot 2^{-k},\label{lm1}\\
&\|U_ng_k\|_{1}\lesssim\|g_k\|_{1}\lesssim  \|f_k\|_{1},\quad n=1,2,\cdots,\label{lm2} \\ 
&|\{U^*g_k(x)>t\}|\lesssim\frac{ \|f_k\|_1}{t},\quad t>4^{-k},\label{a24}
\end{align}
where the last inequality holds in the case of extra conditions and it will be only used in the proof of a.e. convergence.  Using \e {lm1}, we can fix an open set $G$ such that $|G|<\varepsilon$ and 
\begin{equation}\label{a69}
\bigcup_{k\ge 1}\{f_k(x)\neq g_k(x)\}\subset G.
\end{equation}
We claim that $G$ is the desired set that claims the theorem.  Let $f\in L^1(0,1)$ be an arbitrary function. We may suppose that $\|f\|_1=1$. One can see that there exists a subsequence $f_{n_k}$ such that
\begin{align}
&\ll\|\sum_{k=1}^n f_{n_k}-f\rr\|_{1}\to 0,\label{hk1}\\
&\|f_{n_k}\|_1\lesssim 4^{-k},\quad k=1,2,\cdots
\label{hk2}
\end{align}
Denote 
\begin{equation}\label{a42}
g=\sum_{k=1}^\infty g_{n_k}.
\end{equation}
From \e{lm2}, \e{hk2} it follows that 
\begin{equation}\label{a8}
\|g_{n_k}\|_1\lesssim \|f_{n_k}\|_{1}\lesssim 4^{-k} 
\end{equation}
and so the series \e {a42} converges in $L^1$. By \e{a69} we have $\{f(x)\neq g(x)\}\subset G$.  Since $U_n$ is bounded operator on $L^1$,  we have $U_n(g)=\sum_{k=1}^\infty U_n(g_{n_k})$ in the sense  of $L^1$-convergence. Given $\delta>0$ we can fix an integer $l$ such that $4^{-l} <\delta$. Then, by \e {lm2} we will have
\begin{equation}\label{a22}
\sum_{k=l+1}^\infty\|U_n(g_{n_k})\|_1\lesssim \sum_{k=l+1}^\infty\|g_{n_k}\|_1\lesssim \delta .
\end{equation}
On the other hand, applying property (B), for a bigger enough integer $n_0$ we obtain
\begin{equation}\label{a53}
\sum_{k=1}^l\left\|U_n(g_{n_k})-g_{n_k}\right\|_1\le \sum_{k=1}^l\left\|U_n(g_{n_k})-g_{n_k}\right\|_p\le \delta,\quad n\ge n_0.
\end{equation}
Hence, from \e {a8} and \e {a22} for $n\ge n_0$ we get
\begin{equation}\label{a54}
\|U_n(g)-g\|_1\le \sum_{k=1}^l\left\|U_n(g_{n_k})-g_{n_k}\right\|_1+\sum_{k=l+1}^\infty\ll(\left\|U_n(g_{n_k})\right\|_1+\left\|g_{n_k}\right\|_1\rr)\lesssim \delta
\end{equation}
that implies $L^1$ convergence of $U_n(g)$ to $g$. To prove a.e. convergence first note that from property (C) it follows that
\begin{equation}\label{a25}
\lim_{n\to\infty}U_n\left(\sum_{k=1}^mg_{n_k}\right)=\sum_{k=1}^mg_{n_k}\text { a.e.}
\end{equation}
for any fixed $m$. Given numbers $\lambda>0$ denote $t_k=\lambda\cdot  2^{-k-1}$ and chose $m$ satisfying the conditions
\begin{align*}
&2^{-m}<\lambda^2 \text{ and } t_k\ge 4^{-n_k}=\eta_{n_k},\quad k>m.
\end{align*}
In the case of extra conditions (C) and (D), applying \e {a24},\e {hk2}, \e {a8} and \e {a25}, we obtain
\begin{align*}
&\left|\left\{ \limsup_{n\to\infty} |U_n(g)-g|>\lambda \right\}\right|\\
&\qquad=\left|\left\{ \limsup_{n\to\infty} \left|U_n\left(\sum_{k=m+1}^\infty g_{n_k}\right)-\sum_{k=m+1}^\infty g_{n_k}\right|>\lambda \right\}\right|\\
&\qquad\le \left|\left\{ U^*\left(\sum_{k=m+1}^\infty g_{n_k}\right)+\sum_{k=m+1}^\infty |g_{n_k}|>\lambda \right\}\right|\\
&\qquad\le \left|\left\{ \sum_{k=m+1}^\infty U^*\left(g_{n_k}\right)>\lambda/2 \right\}\right|+\left|\left\{ \sum_{k=m+1}^\infty |g_{n_k}|>\lambda/2 \right\}\right|\\
&\qquad\le \sum_{k=m+1}^\infty \left|\left\{U^*\left(g_{n_k}\right)>t_k \right\}\right|+\frac{2}{\lambda}\sum_{k=m+1}^\infty \|g_{n_k}\|_1\\
&\qquad\lesssim \frac{1}{\lambda}\sum_{k=m+1}^\infty2^k\cdot  \|f_{n_k}\|_1+\lambda\\
&\qquad\lesssim \lambda.
\end{align*}
Since $\lambda>0$ can be chosen arbitrarily small, this immediately implies a.e. convergence of $U_n(g)$ and completes the proof of theorem.
\end{proof}
The proof of \trm {T2} is based on the following lemma analogous to \lem {L1}.  Instead of hypothesis  (D) \lem {L2} uses ($\text {D}^*$) and claims the same properties as \lem {L1} does except the inequality $\|U_n(g)\|_1\lesssim \|g\|_1$.
\begin{lemma}\label{L2}
	Let $U\subset [0,1)$ be an everywhere dense open  set and a sequence of bounded linear operators \e {oper} satisfy conditions (A), (B), (C)  and ($\text {D}^*$). Then for any number $0<\eta<1$ and a $\zI$-step function $f(x)$ there is a step function $g(x)$ such that 
	\begin{align}
	&\{g(x)\neq f(x)\}\subset U,\label{a50}\\
	&\|g\|_1\le 2\|f\|_1,\quad n=1,2,\ldots ,\label{a51}\\
	&t\cdot |\{U^*(g)>t\}|\lesssim \|f\|_1,\quad \eta<t<1.\label{a52}
	\end{align}
\end{lemma}
\begin{proof}
	The proof of this lemma is a literal repetition of the proof of \lem {L1} with slight changes. Namely, the sets $G_m$ and so $G_m^{(k)}$ should be chosen from a given open set $U$, according to property $(\text{D}^*)$ instead of (D). So in \e {a45} we will have only upper bound for the measures of  $G_m^{(k)}$, but instead we will have \e {a50}. Hence  the inequalities \e {a14}, \e {a2} and finally \e {a48} will fail, that is why the bound $\|U_n(g)\|_1\lesssim \|g\|_1$ is missing in \lem {L2}. No one from the mentioned inequalities used in the proofs of the other relations of \lem {L1} that are the same relations as in \lem {L2}. With this we can finish the proof of \lem {L2}. 
\end{proof}
\begin{proof}[Proof of \trm {T2}]
	The proof of \trm {T2} is based on \lem {L2} and reflected in the proof of \trm {T1}. Indeed, instead of \e {lm1} we will have the condition $\{g_k(x)\neq f_k(x)\}\subset U$, which will imply $\{g(x)\neq f(x)\}\subset U$. Besides, the lack of condition $\|U_n\|_1\lesssim \|g\|_1$ in \lem {L2} will only affect on $L^1$-convergence property of operators. Namely, we will not have the inequalities \e {a22}, \e {a53} and \e {a54}, but they are not needed in the proof of a.e. convergence. So the \trm {T2} follows.
\end{proof}
\section{Weak type estimates and Fourier series}\label{S3}

\subsection{Trigonometric system}
We shall consider the trigonometric system $\{e^{2\pi inx}\}$ on $[0,1)$.  Denote by $S_n(x,f)$ the partial sums of Fourier series of a function $f$ and let \begin{equation*}
S^*(x,f)=\sup_{n\ge 0}|S_n(x,f)|.
\end{equation*}
For an interval  $\Delta=[a,b)\subset [0,1)$ and an integer $l$ define the partition 
	\begin{equation*}
	\Delta_k=\left[a+(k-1)d,a+kd\right),\quad d=\frac{b-a}{l}\quad k=1,2,\ldots,l.
	\end{equation*}
	Set $\delta_k=[t_k-d\varepsilon/2,t_k+d\varepsilon/2)$, where $t_k=a+(2k-1)d/2$ is the center of $\Delta_k$ and $0<\varepsilon <1$ and denote
	\begin{equation}\label{a65}
	G_l=G_l(\Delta,\varepsilon)=\bigcup_{k=1}^{l}\delta_k.
	\end{equation}
The following proposition shows a weak-$L^1$ inequality for the indicator functions of such sets deriving (D) and ($\textrm{D}^*$) conditions for the partial sum operators of trigonometric Fourier series. 
\begin{prop}\label{P1}
There is an absolute constant $c>0$ such that for any set $G=G_l(\Delta,\varepsilon)$ of the form \e {a65} it holds the inequality
\begin{equation}\label{a37}
|\{S^*(x,\ZI_{G})>\lambda\}|\le c\cdot \frac{|G|}{\lambda},\quad \lambda>0.
\end{equation}
\end{prop}

\begin{proof}[Proof of \pro{P1}]
We shall use the well-known formula	
\begin{equation}\label{a70}
S_n(x,f)= \int_0^1\frac{\sin2\pi n(x-t)}{x-t}f(t)dt+O(\|f\|_1),
\end{equation}
where the integral (that is the modified partial sum) will be denoted by $\tilde S_n(x,f)$.  We also set
\begin{equation*}
 \tilde S^*(f,x)=\sup_{n\ge 1}|\tilde S_n(x,f)|.
\end{equation*}
Using \e {a70}, one can observe that it is enough to prove \e {a37} for $\tilde S^*$ instead of $S^*$.  First let us show \e {a37} when $G$ consists of a single interval $\delta=[\alpha,\beta)$. If $x\in \ZT\setminus\bar\delta$, then we have
	\begin{align}\label{a38}
|\tilde S_n(x,\ZI_\delta)|&\lesssim \int_{ \alpha}^\beta \frac{dt}{|x-t|}\\
&=\ln \left(1+\frac{|\delta|}{\dist(x,\delta)}\right)<\frac{|\delta|}{\dist(x,\delta)},\quad x\in \ZT\setminus\bar \delta.
	\end{align}
From the uniformly boundedness of the integrals $\int_0^b\frac{\sin at}{t}dt$ $a,b>0$, we can conclude
\begin{equation}\label{a39}
|\tilde S_n(x,\ZI_\delta)|\le c,\quad x\in \delta,
\end{equation}
where $c>0$ is an absolute constant.
If $\lambda>c$ then from \e {a38} and \e {a39} we obtain
\begin{align*}
|\{x\in\ZT:\,\tilde S^*(x,\ZI_\delta)>\lambda\}|&=|\{x\in\ZT\setminus \delta:\,\tilde S^*(x,\ZI_\delta)>\lambda\}|\\
&\le |\{x\in\ZT\setminus \delta:\,\frac{|\delta|}{\dist(x,\delta)}>\lambda\}|=\frac{2|\delta|}{\lambda}.
\end{align*}
If $0<\lambda\le c$, then we have
\begin{equation*}
|\{x\in\ZT:\,\tilde S^*(x,\ZI_\delta)>\lambda\}|\le |\delta|+|\{x\in\ZT\setminus \delta:\,\tilde S^*(x,\ZI_\delta)>\lambda\}|\le  \frac{(2+c)|\delta|}{\lambda}
\end{equation*}
which completes the proof of \e{a37} in the single interval case. Now take an arbitrary set $G$ of the form \e {a65}.  Without loss of generality we can suppose $\varepsilon<1/3$. 
 Using the structure of $G$ one can easily check that for $x\in \ZT\setminus\bar \Delta$ we have
\begin{align*}
|\tilde S_n(x,\ZI_G)|&\lesssim \int_{G}\frac{dt}{|x-t|}\\
&\lesssim\varepsilon \int_{\dist(x,\Delta)}^{\dist(x,\Delta)+|\Delta|}\frac{dt}{|x-t|}
\le \frac{\varepsilon |\Delta|}{\dist (x,\Delta)}=\frac{|G|}{\dist (x,\Delta)}.
\end{align*}
Thus we get
\begin{align}\label{a36}
&|\{x\in \ZT\setminus \Delta:\,\tilde S^*(x,\ZI_G)>\lambda \}|\\
&\qquad \lesssim \left|\left\{x\in \ZT\setminus \Delta:\, \dist(x,\Delta)<\frac{|G|}{\lambda}\right\}\right|\lesssim \frac{|G|}{\lambda}.
\end{align}
If $x\in \Delta$, then we have $x\in \Delta_{k(x)}$ for some $k(x)$. We denote $L=\{1,2,\ldots, l\}$.  Splitting the partial sum integral, we write
\begin{align}\label{a75}\\
\tilde S_n(x,\ZI_G)&=\int_{\delta_{k(x)}}\frac{\sin2\pi n(x-t)}{x-t}dt+\sum_{j\in L:\,j\neq k(x)}\int_{\delta_j}\frac{\sin2\pi n(x-t)}{x-t}dt\\
&=u_n(x)+v_n(x).
\end{align}
Applying the inequality in the single interval case,  we get
\begin{align}\label{a76}\\
|\{x\in \Delta:\, \sup_n|u_n(x)|>\lambda \}|&=\sum_{k=1}^l 	|\{x\in\Delta_k:\,\sup_n|u_n(x)|>\lambda \}|\\
&=\sum_{k=1}^l 	|\{x\in\Delta_k:\,\tilde S^*(x,\ZI_{\delta_{k}})|>\lambda \}|\\
&\lesssim\frac{l\varepsilon d }{\lambda}=\frac{|G|}{\lambda}.
\end{align}
For the function $v_n$ we have 
 \begin{align}\label{a71}
 |v_n(x)|&\le  \left|\sum_{j\in L:\,j\neq k(x)}\int_{\delta_j}\frac{e^{2\pi i n (x-t)}}{x-t}dt\right|=\left|\sum_{j\in L:\,j\neq k(x)}\int_{\delta_j}\frac{e^{-2\pi i n t}}{x-t}dt\right|\\
 &=\left|\int_{-\varepsilon d/2}^{\varepsilon d/2}\sum_{j\in L:\,j\neq k(x)}\frac{e^{-2\pi i n (t+t_j)}}{x-t-t_j}dt\right|.
 \end{align}
If $t\in [-\varepsilon d/2,\varepsilon d/2]$, then
\begin{equation}\label{a72}
\sum_{j\in L:\,j\neq k(x)}\left|\frac{1}{x-t-t_j}-\frac{1}{t_{k(x)}-t_j}\right|\lesssim \sum_{j\in L:\,j\neq k(x)}\frac{d}{|t_{k(x)}-t_j|^2}\lesssim \frac{1}{d}.
\end{equation}
Besides, one can uniquely write a decomposition $n=\frac{m}{d}+n^*$, where $m$ is a positive integer and $n^*\in[-1/2d,1/2d)$. By definition we have $t_j=a-d/2+dj$ and so we get
\begin{align*}
e^{-2\pi i nt_j}&=e^{-2\pi i \frac{m}{d} t_j}\cdot e^{-2\pi i n^* t_j}
=e^{-2\pi i \frac{m}{d}(a-d/2)}\cdot e^{-2\pi i mj}\cdot e^{-2\pi i n^* t_j}\\
&=e^{-2\pi i \frac{m}{d}(a-d/2)}\cdot e^{-2\pi i n^* t_j}.
\end{align*}
From this and \e {a72} we obtain
\begin{align}\label{a32}
\left|\sum_{j\in L:\,j\neq k(x)}\frac{e^{-2\pi i n (t+t_j)}}{x-t-t_j}\right|&=\left|\sum_{j\in L:\,j\neq k(x)}\frac{e^{-2\pi i n t_j}}{x-t-t_j}\right|=\left|\sum_{j\in L:\,j\neq k(x)}\frac{e^{-2\pi i n^* t_j}}{x-t-t_j}\right|\\
&\lesssim\left|\sum_{j\in L:\,j\neq k(x)}\frac{e^{-2\pi i n^* t_j}}{t_{k(x)}-t_j}\right|+ \frac{1}{d}\\
&\le \left|\sum_{j\in L:\,0<|j-k(x)|\le \nu(n)}\frac{e^{-2\pi i n^* t_j}}{t_{k(x)}-t_j}\right|\\
&\qquad +\left|\sum_{j\in L:\,|j-k(x)|>\nu(n)\ }\frac{e^{-2\pi i n^* t_j}}{t_{k(x)}-t_j}\right|+\frac{1}{d},
\end{align}
where $\nu(n)=[1/d|n^*|]$.  For the estimation of the second sum recall the well-known inequality 
\begin{equation*}
\left|\sum_{k=0}^ma_ke^{ikx}\right|\le \frac{a_1}{|\sin (x/2)|}
\end{equation*}
that holds whenever $a_1\ge a_2\ge \ldots \ge a_m\ge 0$ and $0<|x|\le \pi$. 
So, using $|n^*|\le 1/2d$ under an additional condition $|n^*|>0$,  we get
\begin{align}\label{a33}
\left|\sum_{j\in L:\,|j-k(x)|> \nu(n)\ }\frac{e^{-2\pi i n^* t_j}}{t_{k(x)}-t_j}\right|&\le\frac{2}{(\nu(n)+1)d}\cdot \frac{1}{|\sin (\pi dn^*)|}\\
&\lesssim \frac{d|n^*|}{d}\cdot \frac{1}{d|n^*|}=\frac{ 1}{d}.
\end{align}
The condition $|n^*|>0$ is not a restriction, since for a smaller enough $|n^*|$ in the left side of \e{a33}  we will have an empty sum. 
To estimate the first sum in \e {a32}, without loss of generality we can suppose that $k(x)\le l/2$ and denote 
\begin{equation*}
r=r(x,n)=\min\{k(x)-1,\nu(n) \}.
\end{equation*}
Observe that $\{j\in\ZZ:\,0<|j-k(x)|\le r(x,n)\}\subset L$ and we get
\begin{align*}
&\left|\sum_{j\in L:\,0<|j-k(x)|\le r(x,n)}\frac{e^{2\pi i n^*(t_{k(x)}- t_j)}}{t_{k(x)}- t_j}\right|\\
&\qquad=\ll|\sum_{j=1}^{r(x,n)}\left(\frac{e^{2\pi i n^*(t_{k(x)}- t_{k(x)-j})}}{t_{k(x)}- t_{k(x)-j}}+\frac{e^{2\pi i n^*(t_{k(x)}- t_{k(x)+j})}}{t_{k(x)}- t_{k(x)+j}}\right)\rr|\\
&\qquad=\ll|\sum_{j=1}^{r(x,n)}\left(\frac{e^{2\pi i n^*(t_{k(x)}- t_{k(x)-j})}}{t_{k(x)}- t_{k(x)-j}}-\frac{e^{-2\pi i n^*(t_{k(x)}- t_{k(x)-j})}}{t_{k(x)}- t_{k(x)-j}}\right)\rr|\\
&\qquad=2\ll|\sum_{j=1}^{r(x,n)}\frac{\sin2\pi  n^*(t_{k(x)}- t_{k(x)-j})}{t_{k(x)}- t_{k(x)-j}}\rr|\\
&\qquad\lesssim r(x,n)\cdot n^*\le\nu(n)\cdot n^*\le \frac{1}{d}.
\end{align*}
Thus we obtain
\begin{align}\label{a34}
 \\ \left|\sum_{j\in L:\,0<|j-k(x)|\le \nu(n)}\frac{e^{-2\pi i n^* t_j}}{t_{k(x)}-t_j}\right|&= \left|\sum_{j\in L:\,0<|j-k(x)|\le  \nu(n)}\frac{e^{2\pi i n^*(t_{k(x)}- t_j)}}{t_{k(x)}- t_j}\right|\\
 &\le \left|\sum_{j\in L:\,0<|j-k(x)|\le r(x,n)}\frac{e^{2\pi i n^*(t_{k(x)}- t_j)}}{t_{k(x)}- t_j}\right|\\
 &\qquad+\sum_{j\in L:\, r(x,n)<|j-k(x)|\le \nu(n)}\frac{1}{|t_{k(x)}- t_j|}\\
 &\lesssim \frac{ 1}{d}+\frac{ 1}{d}\cdot  \ln \frac{\min\{\nu(n),l\}}{r(x,n)+1}.
\end{align} 
Note that in the case $\nu(n)<k(x)$ we have $r(x,n)=\nu(n)$, so in the last estimate we will have just $1/d$. Combining this with \e {a71}, \e {a33} and \e {a34}, we obtain $|v_n(x)|\lesssim \varepsilon$. In the case $k(x)\le \nu(n)$ we have $r(x,n)=k(x)-1$ and so from \e {a71} and \e {a34} we get
\begin{equation}\label{a35}
|v_n(x)|\lesssim \varepsilon+\varepsilon \cdot \ln\frac{l}{r(x,n)+1}\le \varepsilon+\varepsilon \cdot \ln\frac{l}{k(x)}=\gamma(x),
\end{equation}
and hence the inequality \e {a35} holds for every $x\in\Delta$. Thus, for $\lambda>2\varepsilon$ with an appropriate constant $c>0$ we have
\begin{align}\label{a73}\\
|\{x\in\Delta:\, \sup_n|v_n(x)|>c\cdot \lambda\}&|\lesssim |\{x\in\Delta:\, \gamma(x)>\lambda\}|\\
&\lesssim\left|\left\{x\in\Delta:\,\varepsilon \ln\frac{l}{k(x)}>\lambda/2\right\}\right|\\
&=\left|\left\{x\in\Delta:\,k(x)<\frac{l}{e^{\lambda/2\varepsilon}}\right\}\right|\\
&\le\left|\left\{x\in\Delta:\,k(x)<\frac{2l\varepsilon }{\lambda}\right\}\right|\\
&\lesssim \frac{2l\varepsilon}{\lambda}\cdot d\\
&=\frac{2|G|}{\lambda}.
\end{align}
If $\lambda\le 2\varepsilon$, then we will trivially have
\begin{equation}\label{a74}
|\{x\in\Delta:\, \sup_n|v_n(x)|>c\cdot \lambda\}\le |\Delta|=\frac{|G|}{\varepsilon}\le \frac{2|G|}{\lambda}.
\end{equation}
Combining \e {a36}, \e {a75}, \e {a76} with the last estimates \e {a73} and \e {a74} we deduce \e {a37}.  
\end{proof}
\begin{prop}\label{P11}
	The partial sum operators of trigonometric Fourier series satisfy (D) and ($\textrm{D}^*$) conditions. 
\end{prop}
\begin{proof}
To show (D)-condition we chose $\zI$ to be the family of all intervals from $[0,1)$. From \pro {P1} it follows that the set sequence $G_l=G_l([0,1),\varepsilon)$ (see \e{a65}) satisfies condition (D). Now let $U$ be an everywhere dence open set. Observe that for any $l\in\ZN$ there exists a number $\alpha$ such that $t_k=\alpha+k/l\in U$ for each $k=1,2,\ldots, l$. Since $U$ is open, for small enough $0<\delta<\varepsilon$ we will have $G_l=\cup_{k=1}^l[t_k-\delta,t_k+\delta)\subset U$. One can easily check that now the sequence $G_l$ satisfies ($\textrm{D}^*$) condition.
\end{proof}
\subsection{Walsh and Vilenkin systems}
Now consider the Walsh system. We shall use the integral formulas of the partial sums
\begin{equation*}
S_m(x,f)=\sum_{n=0}^{m-1}a_n(f)w_n(x)=\int_0^1D_m(x\oplus t)f(t)dt,
\end{equation*}
where 
\begin{equation*}
D_m(x)=\sum_{k=1}^{m-1}w_k(x)
\end{equation*}
is the Dirichlet kernel and $\oplus$ denotes the dyadic summation. 
We shall suppose that every function and set on $[0,1)$ is $1$-periodically continued. For a set $E\subset [0,1)$ and an integer $n>0$ we shall denote
\begin{equation*}
E(n)=\{x\in [0,1):\, nx\in E\}.
\end{equation*}
The following properties of Dirichlet kernel are well-known (see \cite{GES}):
\begin{align}
&D_{2^n}(x)=\left\{\begin{array}{lcl}
&2^n \hbox{ if }& x\in \Delta_1^{(n)}=[0,2^{-n}),\\
&0\hbox{ if }&x\in[0,1)\setminus \Delta_1^{(n)},
\end{array}
\right.\label{a62}\\
&|D_m(x\oplus y)|\le \frac{1}{|x-y|},\quad x,y\in[0,1).\label{a63}
\end{align}
Let $\phi(x)$ be the function $1/x$ on $[0,1)$ periodically continued. Then, from \e {a63} (with $y=0$) it immediately follows that
\begin{align}\label{a64}
|D_m(2^nx)|\le \phi(2^nx),\quad x\in [0,1),\quad m=1,2,\ldots.
\end{align}
We shall consider the dyadic intervals
\begin{equation*}
\Delta_k^{(n)}=\left[\frac{k-1}{2^n},\frac{k}{2^n}\right), \quad 1\le k\le 2^n,\quad n=0,1,2,\ldots.
\end{equation*}
Given $\Delta=\Delta_k^{(n)}$ and integers $r>0$ and $1\le t\le 2^r$ define the sequense
\begin{equation}\label{a60}
G_l(\Delta,r,t)=\Delta\cap( \Delta_t^{(r)}(2^{l})),\quad l=1,2,\ldots.
\end{equation}
The following statement is the analogous of \pro {P1} for Walsh system.
\begin{prop}\label{P2}
	There is an absolute constant $c>0$ such that for any set $G=G_l(\Delta,r,t)$ of the form \e {a60} it holds the inequality
	\begin{equation}\label{a61}
	|\{S^*(x,\ZI_{G})>\lambda\}|\le c\cdot \frac{|G|}{\lambda},\quad \lambda>0.
	\end{equation}
\end{prop}
\begin{proof}
First observe that $\Delta_t^{(r)}(2^{l})$  consists of $2^l$ component dyadic intervals. Moreover, in each $\Delta_j^{(l)}$, $j=1,2,\ldots,2^l$ there is only one such interval. So in the case $n>l$ the sets $G_l(\Delta,r,t)$ may either be empty or consist of a single dyadic interval. In the first case the inequality \e {a61} is trivial. If $G$ is a dyadic interval say $\Delta_j^{(m)}$, then
\begin{equation*}
S_k(x,\ZI_G)=\left\{\begin{array}{lcl}
&\ZI_G(x)\hbox{ if }& k\ge 2^m,\\
&2^{-m}D_k\left(x\oplus \frac{j-1}{2^m}\right)\hbox{ if }& k<2^m.
\end{array}
\right.
\end{equation*}
So, according to \e {a63} we will have $|S^*(x,\ZI_G)|\le 2^{-m}\cdot |x-(j-1)/2^m|^{-1}$, then
\begin{align*}
	|\{S^*(x,\ZI_{G})>\lambda\}|\le \ll|\ll\{|x-(j-1)/2^m|<\frac{1}{2^{m}\lambda}\rr\}\rr|\le \frac{2}{2^{m}\lambda}=\frac{2|G|}{\lambda},
\end{align*}
 that immediately implies \e {a61}.

Hence, we can consider the case $l\ge n$. Without loss of generality we can suppose that $t=1$ in the definition of the set $G$. Using \e {a60} and \e {a62}, one can easily check that
\begin{equation*}
2^{-n}D_{2^n}\ll(x\oplus\frac{k-1}{2^n}\rr)=\ZI_\Delta(x)
\end{equation*}
and then
\begin{equation*}
\ZI_{G_l}(x)=2^{-n-r}D_{2^r}\ll(2^{l}\cdot x\rr)\cdot D_{2^n}\ll(x\oplus\frac{k-1}{2^n}\rr).
\end{equation*}
Using multiplicative properties of Walsh functions, we get
\begin{align*}
&D_{2^r}\ll(2^{l}\cdot x\rr)=\sum_{j=0}^{2^r-1}w_j\ll( 2^{l}\cdot x\rr)=\sum_{j=0}^{2^r-1}w_{j\cdot 2^{l}}\ll( x\rr),\\
&D_{2^n}\ll(x\oplus\frac{k-1}{2^n}\rr)=\sum_{i=0}^{2^n-1}w_i\ll(x\oplus\frac{k-1}{2^n}\rr)=\sum_{i=0}^{2^n-1}w_i\ll(\frac{k-1}{2^n}\rr)\cdot w_i\ll(x\rr).
\end{align*}
Thus we get
\begin{equation}\label{pol}
\ZI_{G_l}(x)=2^{-n-r}\sum_{j=0}^{2^r-1}\sum_{i=0}^{2^n-1}w_i\ll(\frac{k-1}{2^n}\rr)w_{j\cdot 2^{l}+i}\ll( x\rr).
\end{equation}
For the spectrums of Walsh polynomials from \e{pol} we have
\begin{align*}
\spec\left(\sum_{i=0}^{2^n-1}w_i\ll(\frac{k-1}{2^n}\rr)w_{j\cdot 2^{l}+i}\ll( x\rr)\right)\\
=\{j\cdot 2^{l}, j\cdot 2^{l}+1,\ldots,  j\cdot 2^{l}+2^n-1\}
\end{align*}
and so they are increasing with respect to $j$, since we have $l\ge n$. So for a given integer $m$ each of those spectrums is either wholly or partially included in $[0,m]$. Moreover, at most one of them can be partially included. Using this observation, $m$'th partial sum of $\ZI_{G_l}$ can be split into two sums, collecting the indexes of wholly included spectrums in the first sum and the rest in the second sum. Namely, for some integers $0\le  p<2^r$, $0\le q\le 2^n$ depended on $m$ we will have
\begin{align*}
S_m(x,\ZI_{G_l})&=2^{-n-r}\sum_{j=0}^{p-1}\sum_{i=0}^{2^n-1}w_i\ll(\frac{k-1}{2^n}\rr)w_{j\cdot 2^{l}+i}\ll( x\rr)\\
&\qquad +2^{-n-r}\sum_{i=0}^{q-1}w_i\ll(\frac{k-1}{2^n}\rr)w_{p\cdot 2^{l}+i}\ll( x\rr)\\
&=2^{-n-r}D_{p}\ll(2^{l}\cdot x\rr)\cdot D_{2^n}\ll(x\oplus\frac{k-1}{2^n}\rr)\\
&\qquad +2^{-n-r}\cdot w_{p\cdot 2^{l}}(x)\sum_{i=0}^{q-1}w_i\ll(x\oplus\frac{k-1}{2^n}\rr),\\
&=2^{-n-r}D_{p}\ll(2^{l}\cdot x\rr)\cdot D_{2^n}\ll(x\oplus\frac{k-1}{2^n}\rr)\\
&\qquad +2^{-n-r}\cdot w_{p\cdot 2^{l}}(x)D_q\ll(x\oplus\frac{k-1}{2^n}\rr),
\end{align*}
where we assume $\sum_a^b=0$ wherever $a>b$. So, applying \e {a64} and \e {a62}, we can say
\begin{align}\label{a66}\\
|S_m(x,\ZI_{G_l})|&\le 2^{-r}\cdot \ZI_{\Delta}(x)\cdot \ll|D_{p}\ll(2^{l}\cdot x\rr)\rr|\\
&\qquad +2^{-n-r}\cdot \left|D_q\ll(x\oplus\frac{k-1}{2^n}\rr)\right|\\
&\lesssim 2^{-r}\ZI_{\Delta}(x)\phi(2^{l}\cdot x)+\frac{1}{2^{n+r}|x-(k-1)/2^n|}\\
&=A(x)+B(x),
\end{align}
where $A(x)$ and $B(x)$ are independent of $m$.  According to the definition of function $\phi$ and condition $l\ge n$,  we have
\begin{align}\label{a67}
&|\{x\in [0,1):\,A(x)>\lambda\}|\\
&\qquad =|\{x\in\Delta :\,\phi(2^{l}x)>2^r\lambda\}|= \frac{|\Delta|}{2^r\lambda}=\frac{|G_l|}{\lambda}.
\end{align}
To estimate $B(x)$ we write
\begin{align}\label{a68}
|\{x\in [0,1):\,B(x)>\lambda\}|&=\ll|\ll\{|x-(k-1)/2^n|<\frac{1}{2^{n+r}\lambda}\rr\}\rr|\\
&\le \frac{2}{2^{n+r}\lambda}=\frac{2|G_l|}{\lambda}.
\end{align}
Combining \e {a66}, \e {a67} and \e {a68} we get \e {a61}.
\end{proof}
As in the trigonometric case, \pro{P2} implies 
\begin{prop}\label{P22}
	The partial sum operators of Walsh series satisfy (D) and ($\textrm{D}^*$) conditions. 
\end{prop}
\begin{proof}
	We chose $\zI$ to be the family of all intervals from $[0,1)$. From \pro {P2} it follows that the sequence of sets $G_l=G_l([0,1),r,1)$ defined in \e{a60} and with $r=[\log_2(1/\varepsilon)]+1$ satisfies condition (D). 
	
If $U$ is an everywhere dense open set, then for any integer $l>0$ we find a point $x\in [0,2^{-l})$ such that $V=\{x+j2^{-l}:\,j=0,1,\ldots,2^l-1\}\subset  U$. Obviously, for any $r$ one can find $1\le t\le 2^r$ such that $V\subset \Delta_t^{(r)}(2^{l})$.  Thus, for a bigger enough integers $r$, satisfying also $r<\log_2(1/\varepsilon)$, we can have $G_l=\Delta_t^{(r)}(2^{l})\subset U$.
On the other hand for any dyadic interval $\Delta=\Delta_k^{(n)}$ we have (see \e {a60})
\begin{equation*}
G_l\cap \Delta=G_l(\Delta,r,t)=\Delta\cap \Delta_t^{(r)}(2^{l}).
\end{equation*}
So, according to \pro {P2} these sets satisfy the weak inequality \e {a23} and so we will have condition ($\textrm{D}^*$).
\end{proof}
\begin{remark}
	Analogously, weak type estimate \e{a61}  and so properties (D) and ($\textrm{D}^*$) can be proved also for general Vilenkin systems of bounded type. For the sake of simplicity we restricted us only the consideration of the Walsh system case.
\end{remark}
\bibliographystyle{plain}

\end{document}